\documentclass[reqno,11pt]{amsart}
\usepackage[margin=1.1in, footskip=1cm]{geometry}
\usepackage{amsmath, amsthm, amssymb, amsfonts, booktabs}
\usepackage{mathrsfs}
\usepackage{thmtools, thm-restate}
\allowdisplaybreaks
\pdfpagewidth=\paperwidth
\pdfpageheight=\paperheight
\usepackage[initials,nobysame]{amsrefs}

\newtheorem{theorem}{Theorem}[section]
\newtheorem{lemma}[theorem]{Lemma}

\theoremstyle{definition}

\theoremstyle{remark}

\numberwithin{equation}{section}

\def\calM{{\mathcal M}}

\def\setmin{{\setminus\!\,}}

\def\C{{\mathbb C}}\def\F{{\mathbb F}}
\def\T{{\mathbb T}}
\def\Z{{\mathbb Z}}

\def\grL{{\mathfrak L}}
\def\grm{{\mathfrak m}}\def\grM{{\mathfrak M}}\def\grN{{\mathfrak N}}
\def\grS{{\mathfrak S}}

\def\grL{{\mathfrak L}}

\newcommand{\ord}{\text{ord}}
\newcommand{\tr}{\text{tr}}

\renewcommand{\d}{\,{\rm d}} 

\newcommand{\<}{\begin{equation}}
\renewcommand{\>}{\end{equation}}

\makeatletter
\@namedef{subjclassname@2020}{\textup{2020} Mathematics Subject Classification}
\makeatother

\begin{document}
\title[Moments of Exponential Sums]{On the Moments of Exponential Sums over $r$-Free Polynomials}
\author[Ben Doyle]{Ben Doyle}
\address{Department of Mathematics, Purdue University, 150 N. University Street, West 
Lafayette, IN 47907-2067, USA}
\email{doyle133@purdue.edu}
\thanks{}
\date{}

\begin{abstract} Let $\F_q[t]$ denote the ring of polynomials over the finite field $\F_q$. Building off of techniques of Balog and Ruzsa and of Keil in the integer setting, we determine the precise order of magnitude of $k$th moments of exponential sums over $r$-free polynomials in $\F_q[t]$ for all $k>0$. In the supercritical case $k>1+1/r$, we acquire an asymptotic formula using a function field analogue of the Hardy-Littlewood circle method.  \end{abstract}

\maketitle

\section{Introduction}

For an integer $r\geq 2$, let $a_r(f)$ denote the indicator function for the set of $r$-free polynomials in the ring $\F_q[t]$, where $q=p^h$ for some prime $p$ and positive integer $h$. Just as in the integer setting, one may gain insight into the additive structure of this set by examining the $k$th moments of the associated exponential sum. To define this exponential sum, we require some notation, a more expansive description of which we defer to the next section. We define $\F_q((1/t))$ to be the field of formal Laurent series in $1/t$ over $\F_q$, with elements having the shape $\alpha = \alpha_Mt^M+\cdots +\alpha_0+\alpha_{-1}t^{-1}+\cdots$ for some $M \in \Z$, where $\alpha_i \in \F_q$ for each $i$. We then define an additive character $e(\alpha) = \exp(2\pi i\tr(\alpha_{-1})/p)$ on $\F_q((1/t))$, where $\tr(\cdot):\F_q\rightarrow \F_p$ is the standard trace map given by
\<\label{trace}
\tr(c) = c^p+c^{p^2}+\cdots+c^{p^{h-1}}.
\>
Finally, we define the torus $\T$ to be the set of $\alpha \in \F_q((1/t))$ having $\ord(\alpha)<0$.

If $\d \alpha$ is a Haar measure on $\F_q((1/t))$ normalized so that $\int_\T 1 \d \alpha = 1$, the primary object of study is then the $k$th moment given by
\[I_r(k) = \int_\T \bigg|\sum_{\substack{f \in \calM \\ \deg(f) \leq N}} a_r(f)e(f\alpha)\bigg|^k \d \alpha,\]
where $\calM$ is the set of monic polynomials.

While this construction is perhaps less intuitive than its integer analogue, it makes possible many of the powerful analytic approaches used in additive problems. We prove the following result giving tight upper and lower bounds on $I_r(k)$ for all $k\geq 0$.
\begin{restatable}{theorem}{bounds}\label{bounds}
    Let $k \geq 0$, and let $r \geq 2$ be an integer. Then one has
    \[I_r(k) \asymp \begin{cases}
        q^{Nk/(r+1)} \quad &\text{if } k<1+1/r, \\
        Nq^{N/r} \quad &\text{if } k=1+1/r, \\
        q^{N(k-1)} \quad &\text{if } k>1+1/r,
    \end{cases}\]
    where implicit constants may depend on $k,q$ and $r$.
\end{restatable}
The critical point $k=1+1/r$ is precisely where the orders of magnitude in the super- and sub-critical cases align. It is not hard to see that these magnitudes reflect the contribution of major and minor arcs, respectively. Thus, in the super-critical case, we are able to apply the Hardy-Littlewood circle method to refine this result to an asymptotic formula. We define the zeta function over $\F_q[t]$ to be
\[
\zeta_q(s) = \sum_{f \in \calM} |f|^{-s} = \prod_{\pi \in \calM} \bigg(1-\dfrac{1}{|\pi|^s}\bigg)^{-1} = \dfrac{1}{1-q^{1-s}}.
\]
Here and throughout, we use $\pi$ to denote an irreducible polynomial in $\F_q[t]$.
\begin{restatable}{theorem}{asympt}
    \label{asymp-formula}
    For $k>1+1/r$ and for any $\epsilon>0$, one has the asymptotic formula
    \[\int_\T \bigg|\sum_{\substack{f \in \calM \\ |f| \leq q^N}}a_r(f)e(f\alpha)\bigg|^k \d\alpha = Cq^{N(k-1)} + O\big(q^{N(k-1-\theta)}\big),\]
    where $C$ and $\theta$ are constants depending only on $k,r,$ and $q$.
\end{restatable}
As in other typical applications of the circle method, one may interpret the constant $C$ as a product of local densities at the $\pi$-adic and infinite places. The term $\theta$ will be given explicitly later on, but for now we remark that $\theta$ can be taken to be $1-1/r$ whenever $k>3+1/r$.

In the integer setting, precise information about moments of the exponential sum
\[M(f;\alpha) = \sum_{1 \leq n \leq X} f(n)e^{2\pi i n\alpha}\]
for large $X$ is often highly desirable for Diophantine problems, as it can generally be leveraged for greater control over error terms. Historically, much of the effort has been concentrated on the exponential sums $M(\Lambda;\alpha)$ (see Vaughan \cite{vaughan:1988}), $M(\mu;\alpha)$, and $M(\lambda;\alpha)$ (see Balog and Ruzsa \cites{balog-ruzsa:1998,balog-ruzsa:2001}, or more recently Pandey and Radziwi\l\l \;\cite{pandey-radziwill:2023}), where $\Lambda$, $\mu$, and $\lambda$ are the von Mangoldt, M\"obius, and Liouville functions, respectively. Work of Vaughan and Wooley \cite{vaughan-wooley:1998} also examines the impact of conjectured bounds on moments of the exponential sum for powers, which is relevant for Waring's problem, and more recently work of Pandey and Goldston \cite{goldston-pandey:2019} and Pandey \cites{pandey:2018,pandey:2022} investigates moments of $M(\tau_m;\alpha)$, where $\tau_m$ is the $m$-fold divisor function, for $m\geq 1$.

For an integer $n$, define $\mu_r(n)$ to be equal to 1 when $n$ is $r$-free, and $0$ otherwise. The integer analogue of our object of study would then be the $k$th moment
\[H_r(k) = \int_0^1 |M(\mu_r;\alpha)|^k \d\alpha.\]
Br\"udern, Granville, Perelli, Vaughan, and Wooley \cite{brudern-etal:1998} were the first to investigate this quantity, and for $k=1$ they obtained the bounds
\[X^{1/(2r)} \ll H_r(1) \ll X^{1/(r+1)+\epsilon}.\]
This was improved by Balog and Ruzsa \cite{balog-ruzsa:2001} soon after, obtaining $H_r(1) \asymp X^{1/(r+1)}$. Later, Keil \cite{keil:2013} examined the quantity $H_r(k)$ for all $k\geq 0$ and obtained the bounds 
\begin{alignat*}{3}
    X^{k/(r+1)} \ll &\,H_r(k) \ll X^{k/(r+1)} \quad\quad\; &&\text{if } k < 1+1/r, \\
    X^{1/r}\log X \ll &\,H_r(k) \ll X^{1/r} (\log X)^2\quad\quad\;\; &&\text{if } k=1+1/r, \\
    X^{k-1} \ll &\,H_r(k) \ll X^{k-1} \quad\qquad\;\; &&\text{if } k> 1+1/r,
\end{alignat*}
where implicit constants may depend on $k$ and $r$. Thus, Theorem \ref{bounds} furnishes bounds of equal strength to \cite{keil:2013} when $k \neq 1+1/r$, and a stronger result in the critical case $k = 1+1/r$. 

Our approach to the critical and sub-critical cases are for the most part recreations of the methods of Balog and Ruzsa \cite{balog-ruzsa:2001} and of Keil \cite{keil:2013} in this new setting. However, some segments can be significantly simplified in this setting by considering certain associated exponential sums instead. In particular, the tight upper bound in the critical case relies in part on a reduction to the exponential sum
\[g(N;\alpha) = \sum_{f \in \calM_N} a_r(f)e(f\alpha),\]
where $\calM_N$ is the set of monic polynomials of degree $N$. Similarly, we are able to avoid much of the work required in the integer setting for the sub-critical case by passing to $g(N;\alpha)$.

We note that, somewhat surprisingly, none of the methods are obstructed by the positive characteristic of $\F_q[t]$, as is often the case in function field analogues (see, for example, \cite{liu-wooley:2010} or \cite{effinger-hayes:1991} for discussion on characteristic interference in Weyl differencing methods over $\F_q[t]$). Characteristic obstructions notwithstanding, results in the function field setting often provide insight on what to expect in the integer setting; this is one motivation for working in function fields.

We begin in Section 2 with the introduction of necessary terminology and a pair of preliminary lemmata. Sections 3 and 4 are devoted to the sub-critical and critical cases of Theorem \ref{bounds}, respectively. Finally, in Section 5 we obtain the asymptotic formula of Theorem \ref{asymp-formula}, which in turn implies the super-critical bounds of Theorem \ref{bounds}.

In this paper we use Vinogradov's notation, where $f \ll g$ means that $f(x) \leq Cg(x)$ for some sufficiently large constant $C>0$. Equivalently, we will often use the notation $f = O(g)$. The notation $f = o(g)$ means that for any $C>0$, we have $f(x) < Cg(x)$ when $x$ is large enough. Finally, we use the notation $f \asymp g$ if $f \ll g$ and $g \ll f$, and write $f \sim g$ if $f = g + o(g)$.

\section{Notation and Preliminary Results}

While many of the techniques used in this paper mirror the fairly standard machinery used in the integer setting, there are quite a few subtle details which one must take into account in this new setting. The first task we undertake in this section is thus to introduce the notation and background results necessary for an analysis of the quantity $I_r(k)$. To close the section, we prove a pair of lemmata which supply much of the control needed in the critical and sub-critical cases of Theorem \ref{bounds}.

We write $q = p^h$ for a  fixed prime $p$ and natural number $h$, and fix an integer $r \geq 2$. The variables $d,f,h,\ell$ and $u$ will be reserved throughout for elements of $\F_q[t]$, while we use $\alpha$ and $\beta$ for general elements of $\F_q((1/t))$. We use $\calM_N$ to refer to the set of monic polynomials of degree $N$ in $\F_q[t]$, and $\calM$ to denote the set of all monics in $\F_q[t]$. It will be useful throughout to let $(d,f)$ denote the greatest common (monic) divisor of $d$ and $f$, and similarly to let $[d,f]$ denote the least common (monic) multiple of $d$ and $f$. Recalling that $\alpha \in \F_q((1/t))$ may be written as
\[\alpha = \sum_{i=-\infty}^M \alpha_it^i,\]
for some $M<\infty$, where $\alpha_i \in \F_q$ for each $i$, we can then define $|\alpha| = q^{\ord(\alpha)}$. We adopt the convention that $\ord(0) = -\infty$ and $|0| = 0$. We are then able to define the torus on $\F_q((1/t))$ to be
\[\T = \{\alpha \in \F_q((1/t)): \; |\alpha|<1\}.\]
Just as for the integers, this allows us to uniquely decompose $\alpha \in \F_q((1/t))$ as $\alpha = \lfloor \alpha \rfloor + \{ \alpha \}$, where $\lfloor \alpha \rfloor \in \F_q[t]$ and $\{ \alpha \} \in \T$. We then define $\|\alpha\| = |\{\alpha\}|$. We also make note of the inequality $|\alpha + \beta| \leq \max\{|\alpha|,|\beta|\}$.

We are now prepared to formally define the additive character $e(\cdot)$ for $\F_q((1/t))$. Now for $q = p^h$, we have the non-trivial additive character on $\F_q$ given by $e_q(c) = \exp(2\pi i\tr(c)/p)$, where $\tr(\cdot):\F_q \rightarrow \F_p$ is the standard trace map given as in \eqref{trace}. (Of course here $\pi$ denotes the number, and \emph{not} an irreducible polynomial.)
From this we can induce the desired character $e(\cdot):\F_q((1/t)) \rightarrow \C$ by defining $e(\alpha) = e_q(\alpha_{-1})$. 

The essential orthogonality relation follows easily from the above definition, taking the shape
\[\int_\T e(f\alpha) \d\alpha = \begin{cases}
    0 \quad &\text{when } f \in \F_q[t]\setmin\{0\}, \\
    1 \quad &\text{when } f = 0.
\end{cases}\]
We will also make regular use throughout of Parseval's identity
\[\int_\T \bigg|\sum_{f \in \F_q[t]} c_f e(f\alpha)\bigg|^2 \d\alpha = \sum_{f \in \F_q[t]} |c_f|^2,\]
and the three similar relations
\<\label{sum-ident-non-monic}
    \sum_{|h|<q^m}e(h\alpha) = \begin{cases}
    q^{m} \quad &\text{if } \text{ord}\{\alpha\} < -m, \\
    0 \quad &\text{else}
    \end{cases}
\> 
\<\label{sum-ident-Mn}
    \sum_{h \in \calM_m} e(h\alpha) = \begin{cases}
    e_q(\alpha_{-m-1})q^{m} \quad &\text{if } \text{ord}\{\alpha\} < -m, \\
    0 \quad &\text{else,}
    \end{cases}
\>
and
\<\label{sum-bound-monic}
    \sum_{\substack{h \in \calM \\ |h| \leq q^m}} e(h\alpha) \ll \min\{q^m,\|\alpha\|^{-1}\},
\>
all of which follow quickly from the orthogonality relation and from each other. The precision of \eqref{sum-ident-non-monic} and \eqref{sum-ident-Mn} in comparison to their integer analogues is one of the primary advantages of the function field. For convenience throughout, we define
\[G(\alpha) = \sum_{\substack{f \in \calM \\ |f| \leq q^N}}a_r(f) e(f\alpha),\]
so that $I_r(k) = \int_\T |G(\alpha)|^k \d\alpha$.

In order to apply the circle method in Section 5 we must construct suitable major and minor arcs. It is natural in this problem to consider an adaptation of the usual major and minor arc dissection which focuses on rational approximations with denominators which are $r$th powers. Let
\[\grM_N(f^r,\ell) = \{\alpha \in \T: |f^r\alpha - \ell|<q^{Q-N}\}.\]
Then for a parameter $Q = \delta N$ for $\delta>0$, we define the collection of major arcs $\grM_N = \grM_N(Q)$ to be the union of all arcs $\grM_N(f^r,\ell)$ such that $0 \leq |\ell|<|f|^r \leq q^Q$ with $f$ squarefree and monic, and such that $(\ell,f^r)$ is $r$-free. We then define the associated minor arcs to be $\grm_N = \grm_N(Q) = \T \setmin \grM_N(Q)$, requiring for now only that $\delta <1/2$ to ensure disjoint major arcs. It is a convenient by-product of the polynomial setting that arcs of this shape are either equal or disjoint; this will simplify the application of the circle method. For convenience in Section 3, we also define the set of arcs $\grM(f)$ to be the union of all arcs $\grM(f^r,\ell)$ for a fixed $f$ such that $|\ell|<|f|^r$ and $(\ell,f^r)$ is $r$-free. It is easy to see then that provided $|f|^r\leq q^Q$, the number of arcs in $\grM(f)$ is
\<\label{Phi_r}
\Phi_r(f) = |f|^r \prod_{\substack{\pi \in \calM \\ \pi|f}} \left(1 - \dfrac{1}{|\pi|^r}\right).
\>
We note the standard bounds $|f|^r \ll \Phi_r(f) < |f|^r$. To cap off this considerable selection of notation, we define the standard function field analogue of the M\"obius function, given by
\[\mu(f) = \begin{cases}
    (-1)^t \quad &\text{if } f = \pi_1...\pi_t \text{ and } f \text{ is squarefree}, \\
    0 \quad &\text{else.}
\end{cases}\]
The M\"obius function is exceptionally controlled over $\F_q[t]$, having the exact identity
\<\label{mobius}
\sum_{m \in \calM_n} \mu(m) = \begin{cases}
    1 \quad &{\rm when } \;n = 0,\\
    -q \quad &{\rm when } \;n = 1, \\
    0 \quad &{\rm otherwise}.
\end{cases}
\>

Our principal tool in the subcritical case will be the polynomial analogue of the Fej\'er kernel, given by
\<\label{Fejer}
F(\alpha) = \sum_{|f| < q^N} \bigg(1-\dfrac{|f|}{q^N}\bigg)e(f\alpha) = q^{-N}(q-1)\sum_{k=0}^{N-1} q^k\sum_{|f| \leq q^k} e(f\alpha).
\>
This form of the Fej\'er kernel is clearly non-negative by \eqref{sum-ident-non-monic}, and a quick computation shows that $\int_\T F(\alpha)\d\alpha \leq 1$. As a result, if we define the (non-monic) exponential sums
\<\label{g_1-g_2} 
g_1(\alpha) = \sum_{|f| < q^N} a_r(f)e(f\alpha),\quad g_2(\alpha) = \sum_{|f| < q^N} \bigg(1-\dfrac{|f|}{q^N}\bigg) a_r(f) e(f\alpha),
\>
then we can observe that $g_2(\alpha) = \int_\T F(\alpha-\beta)g_1(\beta)d\beta$ to obtain the useful relation
\<\label{g_2}
\int_\T |g_2(\alpha)|\d\alpha \leq \int_\T \int_\T F(\alpha-\beta)|g_1(\beta)|\d\beta \d\alpha \ll \int_\T |G(\beta)| \d\beta.
\>
It is by this inequality that we will furnish the lower bound in the subcritical case.

The unrestricted monic exponential sum
\[G(\alpha) = \sum_{\substack{0 \leq |f| \leq q^N \\ f \in \calM}} a_r(f)e(f\alpha)\]
is, for the most part, vulnerable to the same approaches as the analogous integer exponential sum. The key advantage in this setting is the ease with which we may convert to simpler exponential sums, in particular the restricted monic exponential sum
\[g(N;\alpha) = \sum_{f \in \calM_N} a_r(f)e(f\alpha)\]
and the non-monic exponential sum $g_1(\alpha)$ defined above. While it is typical to omit the parameter $N$ from the notation, it will be necessary at times to work with $g(n;\alpha)$ for $n<N$, and so we write this parameter explicitly. It is with 
the restriction to $g(N;\alpha)$ that we can obtain the desired upper bounds, and with the extension to $g_1(\alpha)$ that we can obtain the desired lower bounds. This approach may be utilized for a more general class of exponential sums, provided the moments of $g(N;\alpha)$ are sufficiently well-behaved. 

We are now prepared to recreate the necessary preliminary lemmata from \cite{balog-ruzsa:2001} and \cite{keil:2013} in the new setting. This is simplified by working with the restricted exponential sums $g(N;\alpha)$. It is a well-known characterization of $a_r(f)$ that $a_r(f) = \sum_{\substack{d \in \calM \\ d^r|f}}\mu(d)$, and so using this, we define
\[c_i(f) = \sum_{\substack{d \in \calM_i \\ d^r|f}} \mu(d), \quad \text{so that} \quad a_r(f) = \sum_{i=0}^{\lfloor N/r \rfloor} c_i(f).\]
It is obvious that for $i>\deg(f)/r$, we have $c_i(f)=0$.
\begin{lemma}\label{c-ident}
    Let $i$ and $K$ be non-negative integers with $i \leq K/r$. Then one has that
    \[\sum_{f \in \calM_K} |c_i(f)|^2 \ll q^{K+(1-r)i}.\]
\end{lemma}
\begin{proof}
    This is a fairly simple computation which follows in the footsteps of \cite{balog-ruzsa:2001} and \cite{keil:2013}, and so we will be brief. First we note trivially the bound
    \[\sum_{f \in \calM_K} |c_i(f)|^2 \leq \sum_{f \in \calM_K} \bigg(\sum_{\substack{d^r|n \\ d \in \calM_i}} 1\bigg)^2 = \sum_{d_1 \in \calM_i} \sum_{d_2 \in \calM_i} \sum_{\substack{f \in \calM_K \\ [d_1,d_2]^r|f}}1.\]
    Clearly if $q^K<\big|[d_1,d_2]\big|^r$, then the inner sum is 0. If $q^K \geq \big|[d_1,d_2]\big|^r$, the inner sum is equal to $q^K/\big|[d_1,d_2]\big|^r$. Using the trivial bound $|f|^r < \sum_{h|f}|h|^r$ along with the fact that $[d_1,d_2] = d_1d_2/(d_1,d_2)$, we get
    \begin{align*}
        \sum_{f \in \calM_K} |c_i(f)|^2 &\leq \sum_{\substack{d_1,d_2 \in \calM_i \\ \big|[d_1,d_2]\big|^r \leq q^{K}}} \dfrac{q^K}{|d_1^rd_2^r|}\sum_{\substack{h \in \calM \\ h|(d_1,d_2)}}|h|^r.
    \end{align*}
    
    Swapping the order of summation and noting that we can factor the sum over $d_1,d_2$, we get
    \begin{align*}
        \sum_{f \in \calM_K} |c_i(f)|^2 &\leq q^K \sum_{h \in \calM} |h|^r \bigg(\sum_{\substack{d \in \calM_i \\ h|d}}|d|^{-r}\bigg)^2.
    \end{align*}
    Moving a factor of $|h|^{2r}$ inside, we can directly compute
    \[
        \sum_{f \in \calM_K} |c_i(f)|^2 \leq q^K \sum_{h \in \calM} |h|^{-r} \bigg(\sum_{d \in \calM_{i - \deg(h)}} |d|^{-r}\bigg)^2 = q^K \sum_{h \in \calM} |h|^{-r}\bigg(\dfrac{q^{(1-r)i}}{|h|^{1-r}}\bigg)^2.
    \]
    Finally, sorting by degree, we get from the last expression that
    \[
        \sum_{f \in \calM_K} |c_i(f)|^2 \leq q^K \sum_{d=0}^i q^{(r-1)d+(2-2r)i} \ll q^{K+(1-r)i}.
    \]
    This completes the proof.
\end{proof}

The decomposition
\[a_r(f) = \sum_{i=0}^{\lfloor N/r \rfloor} c_i(f)\]
is useful for many of our results, but in order to optimize in the subcritical case, we need a more efficient decomposition. The $c_i(f)$ with $i$ small are small in the $L^1$-norm, while the $c_i(f)$ with $i$ large are small in the $L^2$-norm, and so choosing an appropriate cutoff point and interpolating will give us the optimal bound. For this purpose we let $D = \lfloor N/(r+1)\rfloor$, and define
\<
c_*(f) = \sum_{i>D} c_i(f), \quad \text{so that} \quad a_r(f) = \sum_{i\leq D} c_i(f) + c_*(f).
\>
In the same spirit, define
\[T_i(\alpha) = \sum_{f \in \calM_N} c_i(f)e(f\alpha) = \sum_{f \in \calM_N}\sum_{\substack{d \in \calM_i \\ d^r | f}} \mu(d)e(f\alpha)
\quad\text{and}\quad
T_*(\alpha) = \sum_{f \in \calM_N} c_*(f)e(f\alpha),\]
so that we have the decompositions
\<\label{T_*}
g(N;\alpha) = \sum_{i=1}^{\lfloor N/r \rfloor} T_i(\alpha) = \sum_{i=1}^{D} T_i(\alpha) + T_*(\alpha).
\>
\begin{lemma}\label{T_i}
    Let $0 \leq i \leq \lfloor N/r \rfloor$ and $k \in [1,2]$. One has that
    \[\int_\T |T_i(\alpha)|^k \d\alpha \ll q^{i+(N-ri)(k-1)}\]
    and as a result,
    \[\int_\T |T_*(\alpha)|^2 \d\alpha \ll q^{N+(1-r)D}.\]
\end{lemma}
\begin{proof}
    We prove the first result by first obtaining the $L^1$- and $L^2$-bounds and then interpolating. The $L^2$-bound follows from Parseval's identity and the previous lemma. We have
    \[
        \int_\T |T_i(\alpha)|^2\d\alpha = \int_\T \bigg|\sum_{f \in \calM_N} c_i(f)e(f\alpha)\bigg|^2 \d\alpha = \sum_{f \in \calM_N} |c_i(f)|^2 
    \]
    by Parseval's identity. Applying the previous lemma, we get
    \[
        \int_\T |T_i(\alpha)|^2 \d\alpha \ll q^{N+(1-r)i}
    \]
    so long as $0 \leq i \leq N/r$. For the $L^1$-bound, we use the other representation of $T_i$ to get
    \[\int_\T|T_i(\alpha)| \d\alpha = \int_\T \bigg|\sum_{d \in \calM_i} \mu(d)\sum_{h \in \calM_{N-ri}} e(hd^r\alpha)\bigg|\d\alpha.\]
    
    Recalling \eqref{sum-ident-Mn}, we see that that the inner sum is non-zero only when $\alpha$ is an element of a suitable arc $\grM_N(d^r,\ell)$ for $\ell \in \F_q[t]$ with $|\ell| < |d|^r$. While the major arcs are only disjoint for $Q$ sufficiently small, we require little precision for this result, and so we use the inequality
    \[\int_\T |T_i(\alpha)| \d\alpha \leq \sum_{d \in \calM_i} \sum_{\substack{|\ell| < q^{ri} \\ (\ell,d^r)\;r{\rm -free}}}\int_{\grM_N(d^r,\ell)} \bigg|\sum_{h \in \calM_{N-ri}} e(hd^r \alpha)\bigg|\d\alpha.\]
    By \eqref{sum-ident-Mn}, we get
    \begin{align*}
        \int_\T|T_i(\alpha)| \d\alpha &\ll \sum_{d \in \calM_i}\sum_{|\ell| < q^{ri}}\int_{\grM(d^r,\ell)} q^{N-ri}\d\alpha \ll q^i.
    \end{align*}
    This is the desired bound for $k=1$. For $1 < k < 2$, we can interpolate the $L^1$- and $L^2$-bounds using H\"older's inequality to get
    \[\int_\T |T_i(\alpha)|^k \d\alpha \ll q^{i+N(k-1)-ri(k-1)}.\]
    Finally, we note that by Parseval's identity,
    \[\int_\T |T_*(\alpha)|^2 \d\alpha = \sum_{i>D}\int_\T |T_i(\alpha)|^2 \d\alpha \ll \sum_{i>D} q^{N+(1-r)i} \ll q^{N+(1-r)D}.\]
    This completes the proof.
\end{proof}

We observe that Lemma \ref{T_i} is essentially the analogue of Lemma 2.2 in \cite{keil:2013} on a dyadic restriction of the exponential sum. This restriction allows us to handle the following upper bound without much of the sophisticated machinery needed in the integer setting.

\section{The Sub-critical Case $k<1+1/r$}

 We remark again that, outside of this additional decomposition, the strategies of this section are essentially direct translations of those of \cite{balog-ruzsa:2001} and \cite{keil:2013} into the function field setting.
\begin{lemma}\label{sub-upper}
    For $k < 1+1/r$, one has
    \[\int_\T |G(\alpha)|^k \d\alpha \ll q^{Nk/(r+1)}.\]
\end{lemma}
\begin{proof}
    First consider the restricted exponential sum $g(N;\alpha)$. We begin with the case $1 < k < 1+1/r$, which will imply the full result by an application of H\"older's inequality. To achieve the bound in this case, we decompose our exponential sum as in (\ref{T_*}). A full decomposition according to degree would be inefficient at the tail end, as the $T_i$ for larger $i$ have especially small $L^2$ norms. For the $T_i$ with $i\leq D$, we apply a polynomial weight as in the proof of Theorem \ref{bounds} in order to save a logarithmic factor. We write
    \begin{align*}
        \int_\T |g(N;\alpha)|^k \d\alpha &= \int_\T \bigg|\sum_{i=0}^{D} (1+D-i)^{-1}(1+D-i)T_i(\alpha)+ T_*(\alpha)\bigg|^k \d\alpha \\
        &\ll \bigg(\sum_{i=0}^D (1+D-i)^{-\frac{k}{k-1}}\bigg)^{\frac{1}{k-1}}\int_\T\sum_{i=0}^{D} (1+D-i)^k|T_i(\alpha)|^k\d\alpha + \int_\T | T_*(\alpha)|^k \d\alpha.
    \end{align*}
    The first sum is bounded since $k>1$, and so applying H\"older's inequality to the $T_*(\alpha)$ term, we get
    \[\int_\T |g(N;\alpha)|^k \d\alpha \ll \sum_{i=0}^{D} (1+D-i)^k\int_\T|T_i(\alpha)|^k\d\alpha + \bigg(\int_\T | T_*(\alpha)|^2 \d\alpha\bigg)^{k/2}.\]
    
    Now we can apply the results of Lemma \ref{T_i} to get
    \[\int_\T |g(N;\alpha)|^k \d\alpha \ll q^{N(k-1)}\sum_{i=0}^{D} (1+D-i)^kq^{i(1-r(k-1))} + \big(q^{N+(1-r)D}\big)^{k/2}.\]
    Since $k<1+1/r$, the exponential factor in the above sum has positive power, and so
    \[q^{N(k-1)}\sum_{i=0}^{D} (1+D-i)^k q^{i(1-r(k-1))} \ll q^{N(k-1)}q^{D(1-r(k-1))} \ll q^{Nk/(r+1)}.\]
    But we also have that
    \[\big(q^{N+(1-r)D}\big)^{k/2} = \big(q^{2N/(r+1)}\big)^{k/2} = q^{Nk/(r+1)},\]
    and so we obtain the desired result for $1<k<1+1/r$. For $k\leq 1$, we simply observe that if $0<\delta<1/r$, then
    \[\int_\T |g(N;\alpha)|^k \d\alpha \ll \bigg(\int_\T |g(N;\alpha)|^{1+\delta} \d\alpha\bigg)^{k/(1+\delta)} \ll q^{Nk/(r+1)}.\]

    The same argument can be repeated with the decomposition $G(\alpha) = \sum_{n \leq N} g(n;\alpha)$ to obtain the desired result. Apply H\"older for the case $k>1$ to observe that
    \[\int_\T |G(\alpha)|^k \d\alpha \ll \sum_{1\leq n \leq N} (N+1-n)^k\int_\T |g(n;\alpha)|^k \d\alpha \ll q^{Nk/(r+1)}.\]
    For $k\leq 1$, H\"older interpolation again implies the result.
\end{proof}

The effort to obtain the lower bound is much more involved. We work with the weighted sum $g_2(\alpha)$ defined in \eqref{g_1-g_2} which is easier to control, and then apply \eqref{g_2} to obtain the lower bound. A decomposition of $g_2$ similar to the $T_i$ decomposition is used here. Define the averaging function
\[F_h(\alpha) = \sum_{\substack{|f|\leq q^N \\ h|f}}\bigg(1-\dfrac{|f|}{q^N}\bigg)e(f\alpha) = \dfrac{1}{|h|} \sum_{|a| < |h|} F\bigg(\alpha - \dfrac{a}{h}\bigg),\]
so that we have the decomposition
\begin{align*}
    g_2(\alpha) &= \sum_{|f| \leq q^N} \bigg(1-\dfrac{|f|}{q^N}\bigg)\sum_{i<D} c_i(f)e(f\alpha) + \sum_{|f| \leq q^N} \bigg(1-\dfrac{|f|}{q^N}\bigg)c_*(f) e(f\alpha) \\
    &= \sum_{|d| \leq q^D} \mu(d)F_{d^r}(\alpha) - \sum_{|d| \leq q^D}\mu(d) + \sum_{|f| \leq q^N} \bigg(1-\dfrac{|f|}{q^N}\bigg)c_*(f) e(f\alpha).
\end{align*}
The middle sum compensates for the terms associated to $f = 0$ in $F_h(\alpha)$, and is $O(1)$ in the setting of $\F_q[t]$ by \eqref{mobius}. 

Label the first sum by $R(\alpha)$. It is reasonable to expect that the last term is of smaller order on most of $\T$, as $c_*(f)$ is only non-zero for large $f$, in which case the smoothing factor is small. Thus our goal is to show that on a restricted portion of $\T$, the first sum is the desired order of magnitude, and the last sum is of smaller order. 

Denote the first sum by $R(\alpha)$. Substituting the definition of $F_{d^r}(\alpha)$, we get
\[R(\alpha) = \sum_{|d| \leq q^D} \dfrac{\mu(d)}{|d|^r}\sum_{|a| < |d|^r} F\bigg(\alpha - \dfrac{a}{d^r}\bigg).\]
We may sort the inner sum by the largest degree monic $r$-th power dividing both $a$ and $d^r$, which we label $m^r$. Noting that $\mu(d)$ is supported on squarefrees, and then reindexing $d/m$ as $d$, we get
\[R(\alpha) = \sum_{|d| \leq q^D}\dfrac{\mu(d)}{|d|^r}\sum_{\substack{|a| < |d|^r \\ (a,d^r) \; r\text{-free}}}F\bigg(\alpha - \dfrac{a}{d^r}\bigg)\sum_{\substack{|m| \leq q^D/|d| \\ (m,d) = 1 \\ m \in \calM}} \dfrac{\mu(m)}{|m|^r}.\]
We write
\[b_d = \sum_{\substack{|m| \leq q^D/|d| \\ (m,d)=1 \\ m \in \calM}}\dfrac{\mu(m)}{|m|^r} \quad {\rm and} \quad H_d(\alpha) = \dfrac{1}{|d|^r}\sum_{\substack{|a| < |d|^r \\ (a,d^r) \; r\text{-free}}}F\bigg(\alpha - \dfrac{a}{d^r}\bigg).\]
so that
\[R(\alpha) = \sum_{|d|\leq q^D}\mu(d)b_dH_d(\alpha).\]

The factor $b_d$ is well-known, and is equal to $L(r,\chi_0)^{-1} + O((|d|/q^D)^{r-1})$, where $\chi_0$ is the trivial Dirichlet character modulo $d$. To achieve sufficiently tight upper and lower bounds, first note that \eqref{mobius} implies that for $i>1$ at most $(1/2)q^i$ of the polynomials $m$ of degree $i$ have $\mu(m) = 1$. Thus it is easy to deduce that
\<
b_d \leq 1+\dfrac{1}{2}(\zeta_q(r)-\zeta_q(r)^{-1}) \leq 7/4.
\>
To obtain the lower bound, we observe that
\[b_d \geq 1 - \sum_{\substack{m \in \calM \\ |m| > 1 \\ (m,d) = 1}}\dfrac{1}{|m|^r} = 2 - L(r,\chi_0) = 2 - \zeta_q(r)\prod_{\pi | d} (1-|\pi|^{-r}).\]
Provided $d \neq 1$, the product will be at least $3/4$, and so $b_d \geq 1/2$. If $d = 1$ then $b_d \rightarrow \zeta_q(r)^{-1}$, which is at least $1/2$. For convenience, we weaken this slightly so that we have the upper and lower bounds
\<\label{b_d}
1/4 \leq b_d \leq 7/4.
\>

Now we need to understand the contribution of $H_d(\alpha)$. First note that applying \eqref{sum-ident-non-monic} to the second form of $F(\alpha)$ in \eqref{Fejer}, we have whenever $\ord\{\alpha\}<-N$ that 
\[F(\alpha) = q^{1-N}(q-1)\sum_{i=0}^{N-1} q^{2i} = q^{1-N}\bigg(\dfrac{q^{2(N-1)}-1}{q+1}\bigg) > q^N-q.\]
Thus we have for $\alpha \in \grM_N(d^r,\ell)$ that $H_d(\alpha) > (q^N-q)/|d|^r$. In order to prove the desired lower bound for $R(\alpha)$, we need to show that this contribution is not canceled by contributions from the other summands in $R(\alpha)$. While this may not always be the case, we find that we can restrict to a smaller region where this is satisfied, and then show that this smaller region is not too small. We define this region to be
\[\grN(d) = \bigg\{\alpha \in \grM(d): \sum_{\substack{|d'| \leq q^D \\ d' \neq d}}H_{d'}(\alpha)\leq \dfrac{q^{N - 5}}{|d|^r}\bigg\}.\]

While the major arcs $\grM_N(d^r,\ell)$ are only disjoint when we restrict to $|d|< q^{N/(2r)}$, the sets $\grN(d)$ are certainly disjoint for all $d$, since again the $H_d(\alpha)$ are non-negative and $H_d(\alpha) \geq (q^{N-1}-1)/|d|^r$ for $\alpha \in \grN(d) \subseteq \grM(d)$. If $\alpha \in \grN(d)$, we have by \eqref{b_d} that
\[|R(\alpha)|\geq b_d H_d(\alpha) - \sum_{\substack{|d'| \leq q^D \\ d' \neq d}} b_{d'}H_{d'}(\alpha) \geq \dfrac{q^N-q}{4|d|^r} - \dfrac{7q^{N-5}}{4|d|^r} \geq C\dfrac{q^N}{|d|^r}\]
for some positive constant $C$ (bounded below by, say, $1/128$).

Now we wish to give a lower bound on $|\grN(d)|$. First note that for each $d$,
\[|\grM(d)| = q^{-N}\Phi_r(d),\]
and so $|d|^r/q^N \ll |\grM(d)| < |d|^r/q^N$. We want to show that $|\grM(d) \setmin \grN(d)| = o(|d|^r/q^N)$, and to do so we prove the following quasi-orthogonality relation for $H_d(\alpha)$.
\begin{lemma}
    For $|d_1|,|d_2|\leq q^D$ and $d_1 \neq d_2$, one has
    \[\int_\T H_{d_1}(\alpha)H_{d_2}(\alpha)\d\alpha \ll 1.\]
\end{lemma}
\begin{proof}
    Let $\ord\{\alpha\} = -K$. We begin by noting that again applying \eqref{sum-ident-non-monic} to the second form of $F(\alpha)$ in \eqref{Fejer}, we can obtain the upper bound
    \[F(\alpha) = q^{-N}(1-1/q)\sum_{i=0}^{\min\{N-1,K-1\}}q^{2i+1} \ll \min\{q^N,\|\alpha\|^{-2}q^{-N}\}.\]
    Because of this, we can deduce that
    \begin{align*}
        F(\alpha)F(\beta) &\leq (F(\alpha)+F(\beta))\min\{F(\alpha),F(\beta)\} \\
        &\ll (F(\alpha)+F(\beta))\min\{q^N,\|\alpha\|^{-2}q^{-N},\|\beta\|^{-2}q^{-N}\} \\
        &\leq (F(\alpha)+F(\beta))\min\{q^N,\|\alpha - \beta\|^{-2}q^{-N}\},
    \end{align*}
    and so in particular,
    \[\int_\T F(\beta)F(\alpha+\beta)\d\beta \ll \min\{q^N,\|\alpha\|^{-2}q^{-N}\}.\]
    
    Now using the definition of $H_d(\alpha)$ and applying this bound, we get
    \begin{align*}
        \int_\T H_{d_1}(\alpha)H_{d_2}(\alpha)\d\alpha &= \dfrac{1}{|d_1d_2|^r} \sum_{\substack{|a_1| < |d_1|^r \\ (a_1,d_1^r)\;r\text{-free}}} \sum_{\substack{|a_2| < |d_2|^r \\ (a_2,d_2^r)\; r\text{-free}}}\int_\T F\bigg(\alpha - \dfrac{a_1}{d_1^r}\bigg)F\bigg(\alpha - \dfrac{a_2}{d_2^r}\bigg)\d \alpha \\
        &\ll \dfrac{1}{|d_1d_2|^r} \sum_{\substack{|a_1| < |d_1|^r \\ (a_1,d_1^r)\;r\text{-free}}} \sum_{\substack{|a_2| < |d_2|^r \\ (a_2,d_2^r)\; r\text{-free}}} \min\bigg\{q^N, \bigg\|\dfrac{a_2}{d_2^r}-\dfrac{a_1}{d_1^r}\bigg\|^{-2}q^{-N}\bigg\}.
    \end{align*}
    Note that since $|a_1|<|d_1|^r$ and $|a_2|<|d_2|^r$, we have
    \[\bigg\|\dfrac{a_2}{d_2^r}-\dfrac{a_1}{d_1^r}\bigg\| = \bigg|\dfrac{a_2}{d_2^r}-\dfrac{a_1}{d_1^r}\bigg|.\]
    Define $m = a_2d_1^r - a_1d_2^r$. In order for the second term in our minimum to dominate, we need
    \[\bigg|\dfrac{a_2}{d_2^r}-\dfrac{a_1}{d_1^r}\bigg| = \dfrac{|a_2d_1^r-a_1d_2^r|}{|d_1d_2|^r} = \dfrac{|m|}{|d_1d_2|^r} < q^{-N},\]
    or equivalently, we require $\grM_N(d_1^r,a_1) = \grM_N(d_2^r,a_2)$. We make a few observations regarding the quantity $m$. Since $d_1 \neq d_2$, and $(a_1,d_1^r)$ and $(a_2,d_2^r)$ are $r$-free, the quantity $m = m(a_1,a_2)$ will never be zero. We also note that given a fixed $a_1$ and $m$, there is at most one $a_2$ such that $m(a_1,a_2) = m$. Finally, if $|[d_1,d_2]|<q^{N/r}$, then we have no solutions to the above inequality. 
    
    By Bezout's identity, we can find a unique pair $a_1',a_2'$ such that $|a_1'|<|d_1|$ and $|a_2'|<|d_2|$, and $m(a_1',a_2') = (d_1,d_2)^r$. If $|[d_1,d_2]|\geq q^{N/r}$, then for each $f$ with $\deg(f) \leq r\deg[d_1,d_2]-N$, the pair $(fa_1',fa_2')$ satisfies the inequality. There are $\big|[d_1,d_2]\big|^r/q^N$ such $f$. Thus there are $\big|[d_1,d_2]\big|^r/q^N$ pairs $(a_1,a_2)$ in our double sum for which the term $q^N$ dominates, and we note for the other term that
    \[\bigg\|\dfrac{a_2}{d_2^r}-\dfrac{a_1}{d_1^r}\bigg\|^{-2}q^{-N} = \dfrac{\big|[d_1,d_2]\big|^{2r}}{q^N|m|^2}.\]
    
    Returning to our bound, then, we get
    \[\int_\T H_{d_1}(\alpha)H_{d_2}(\alpha)\d\alpha \ll \dfrac{1}{|d_1d_2|^r}\bigg(\big|[d_1,d_2]\big|^r+\dfrac{\big|[d_1,d_2]\big|^{2r}}{q^N}\sum_{\substack{a_1,a_2 \\ m\geq\big|[d_1,d_2]\big|^rq^{-N}}}|m|^{-2}\bigg).\]
    Finally, using that there are at most $q^N$ pairs $(a_1,a_2)$ which yield a particular $m$ along with the fact that $\zeta_q(2)$ is finite, we obtain the desired result.
\end{proof}

This gives us the desired size constraint. Let $\grL(d) = \grM(d)\setmin \grN(d)$. Then simply by definition, we have that
\[\int_{\grL(d)} H_d(\alpha)\sum_{\substack{|d'| \leq q^D \\ d' \neq d}}H_{d'}(\alpha) d\alpha \geq |\grL(d)|\dfrac{q^{2N}}{q^5|d|^{2r}}.\]
The above lemma of course yields the upper bound
\[\int_{\grL(d)} H_d(\alpha)\sum_{\substack{|d'| \leq q^D \\ d' \neq d}}H_{d'}(\alpha) d\alpha \leq \sum_{\substack{|d'| \leq q^D \\ d' \neq d}} \int_\T H_d(\alpha)H_{d'}(\alpha)d\alpha \ll q^D,\]
and so together we can deduce that $|\grL(d)| \ll q^{D-2N}|d|^{2r}$. So long as $|d|<q^D$, then, we have that $|\grN(d)| \gg |\grM(d)|$. While the condition $|d|<q^D$ is be sufficient for proving that $R(\alpha)$ is large, we use a stronger condition to ensure that our third sum in the decomposition of $g_2(\alpha)$ is sufficiently small. 

Let $\epsilon>0$, and define $\grN$ to be the union of all $\grN(d)$ for $d$ monic with $|d|\leq \epsilon q^D$. Using the upper bound $|\grM(d)| \leq |d|^r/q^N$ and that $\grN(d) \subseteq \grM(d)$, we get that
\[|\grN| \leq \sum_{|d|\leq \epsilon q^D}|d|^r/q^N \leq \epsilon^r q^{Dr-N} \leq \epsilon^r.\]
Now applying the pointwise bound on $R(\alpha)$ along with the lower bound on $|\grN|$, we can deduce that
\<\label{R(a)}
\int_\grN |R(\alpha)|\d\alpha \geq C\sum_{|d| \leq \epsilon q^D}\dfrac{|\grN(d)|}{|d|^r} \gg \epsilon q^D.
\>
We put all of this to use in the proof of the desired result. 
\begin{lemma}\label{L1-lower}
    For $k<1+1/r$, one has
    \[\int_\T |g_1(\alpha)|^k \d\alpha \gg q^{Nk/(r+1)}.\]
\end{lemma}
\begin{proof}
    First suppose $k = 1$, and consider $g_2(\alpha)$. We have from our decomposition that
    \[|g_2(\alpha)| \geq |R(\alpha)| - \bigg|\sum_{|f| \leq q^N} \bigg(1-\dfrac{|f|}{q^N}\bigg)c_*(f) e(f\alpha)\bigg| + O(1).\]
    Integrating over $\grN$ and applying \eqref{R(a)} along with Cauchy-Schwarz, we find that
    \[\int_\grN |g_2(\alpha)|\d\alpha \gg \epsilon q^D - \bigg(|\grN|\sum_{k\geq D} \sum_{q^D \leq |f| \leq q^N} |c_k(f)|^2\bigg)^{1/2} + O(1).\]
    Finally, applying Lemma \ref{c-ident} and \eqref{g_2}, we get that
    \[\int_\T |g_1(\alpha)| \d \alpha \gg \int_\T |g_2(\alpha)| \d \alpha \gg \int_\grN |g_2(\alpha)| \d \alpha \gg q^D(\epsilon - \epsilon^r) \gg q^D.\]
    This proves the result for $k = 1$. For $k>1$, the result is implied directly by the $k=1$ case using H\"older's inequality. Finally, for $k<1$, we write the $L^1$-mean as an interpolation between $k$ and $1+1/(r+1)$. 
    
    We have
    \[q^{N/(r+1)} \ll \int_\T|g_1(\alpha)|\d\alpha \ll \bigg(\int_\T |g_1(\alpha)|^k\d\alpha\bigg)^{\theta}\bigg(\int_\T |g_1(\alpha)|^{1+1/(r+1)}\d\alpha\bigg)^{1-\theta},\]
    where $1 = \theta k + (1-\theta)(1+(r+1)^{-1})$. Since $1+1/(r+1)<1+1/r$, we can apply the upper bound in Theorem \ref{sub-upper} to get
    \[\bigg(\int_\T |g_1(\alpha)|^k\d\alpha\bigg)^{\theta}\gg q^{N/(r+1)}\bigg(q^{N(r+2)/(r+1)^2}\bigg)^{\theta-1}.\]
    Raising to the power $1/\theta$, we obtain the desired result.
\end{proof}

\section{The Critical Case}

In this section we briefly examine the critical case $k=1+1/r$. The function field setting gives a slight advantage over the integer setting in this case, which can be attributed to the precision of \eqref{sum-ident-Mn}. Although the sum in the integer setting can be dissected into dyadic intervals in the same way, the exponential sum bound
\[\bigg|\sum_{X \leq x \leq X+Y} e(\alpha x)\bigg| \ll \min\{Y,\|\alpha\|^{-1}\}\]
does not easily gain from this dissection. An application of Keil's original method immediately yields an upper bound which is $O(N^{2-1/r}q^{N/r})$ because of this extra savings. We can improve this to a tight upper bound by being a bit more careful. Define 
\[h_j(\alpha) = \sum_{i \leq j} T_i(\alpha) \qquad {\rm and } \qquad H_j(\alpha) = \sum_{i>j} T_i(\alpha).\]
We apply Parseval's identity just as in Lemma \ref{T_i} to obtain the bounds
\<\label{h_j-bounds}
\int_\T |h_j(\alpha)|\d\alpha \ll q^j \qquad {\rm and } \qquad \int_\T |H_j(\alpha)|^2 \d\alpha \ll q^{N+(1-r)j}.
\>
\begin{lemma}\label{crit-upper}
    One has the upper bound
    \[\int_\T |G(\alpha)|^{1+1/r} \d\alpha \ll Nq^{N/r}.\]
\end{lemma}
\begin{proof}
    The idea is to decompose a single factor of $|g(N;\alpha)|$, and separate the remaining factor into $h_j(\alpha)+H_j(\alpha)$ for each degree $j$. Then we can apply the same reasoning from Lemma \ref{sub-upper} to obtain the bound for $G(\alpha)$. This approach allows us to save the additional factor of $N$ which arises in the traditional method. Using the standard bound $|x+y|^k \ll |x|^k+|y|^k$, we write
    \begin{align*}
        \int_\T |g(N;\alpha)|^{1+1/r}\d\alpha &\leq \sum_{j=0}^{\lfloor N/r \rfloor} \int_\T |T_j(\alpha)||h_j(\alpha)+H_j(\alpha)|^{1/r}\d\alpha \\
        &\ll_r \sum_{j=0}^{\lfloor N/r\rfloor} \bigg(\int_\T |T_j(\alpha)||h_j(\alpha)|^{1/r}\d\alpha + \int_\T |T_j(\alpha)||H_j(\alpha)|^{1/r}\d\alpha\bigg).
    \end{align*}
    
    To utilize the $L^1$ control of $h_j(\alpha)$ and the $L^2$ control of $H_j(\alpha)$, we apply H\"older's inequality to the above decomposition to get
    \begin{align*}
        \int_\T |g(N;\alpha)|^{1+1/r}\d\alpha &\ll \sum_{j=0}^{\lfloor N/r\rfloor} \bigg(\int_\T |T_j(\alpha)|^2\d\alpha\bigg)^{\frac{1}{r}}\bigg(\int_\T |T_j(\alpha)|\d\alpha\bigg)^{1-\frac{2}{r}}\bigg(\int_\T|h_j(\alpha)|\d\alpha\bigg)^{\frac{1}{r}} \\
        &\qquad + \bigg(\int_\T |T_j(\alpha)|^2d\alpha\bigg)^{\frac{1}{2r}}\bigg(\int_\T |T_j(\alpha)|d\alpha\bigg)^{1-\frac{1}{r}}\bigg(\int_\T|H_j(\alpha)|^{2}d\alpha\bigg)^{\frac{1}{2r}}.
    \end{align*}
    Applying the bounds from Lemma \ref{T_i} and \eqref{h_j-bounds}, we get that this simplifies to
    \[\int_\T |g(N;\alpha)|^{1+1/r}\d\alpha \ll \sum_{j=0}^{\lfloor N/r \rfloor} (q^{N/r}+q^{N/r}) \ll Nq^{N/r}.\]
    We then use the same polynomial weighting argument from Lemma \ref{T_i} to show that
    \[\int_\T |G(\alpha)|^{1+1/r} \d\alpha \ll \int_\T |g(N;\alpha)|^{1+1/r} \d\alpha \ll Nq^{N/r}. \qedhere\]
\end{proof}

The lower bound follows quickly from the major arc analysis in Section 5.
\begin{lemma}
    One has the lower bound
    \[\int_\T |G(\alpha)|^{1+1/r} \d\alpha \gg Nq^{N/r}.\]
\end{lemma}
\begin{proof}
    We have that
    \[\int_\T |G(\alpha)|^{1+1/r}\d\alpha \geq \int_{\grM_N(Q)} |G(\alpha)|^{1+1/r} \d\alpha \gg \grS_r(1+1/r;Q)q^{N/r} + O(q^{N/r}),\]
    We have that $\grS_r(k;Q) \gg N$ by Lemma \ref{sing-ser-conv}, and so we have the desired result.
\end{proof}

\section{An Asymptotic Formula}

While appearing fleetingly in the previous sections, it is in proving the asymptotic formula of Theorem \ref{asymp-formula} that the mechanisms of the Hardy-Littlewood circle method come into focus. Recall the major arcs $\grM_N(f^r, \ell)$, where $0\leq |\ell|<|f|^r \leq q^Q$ with $f$ squarefree and $(\ell,f^r)$ $r$-free, and recall that we take $Q<N/2$ to ensure that the major arcs are disjoint. We define
\[S(f^r,\ell) = \sum_{d \in \calM} \dfrac{\mu(d)}{|d|^r} \sum_{|u| < |f|^r}e(ud^r\ell/f^r) \quad \text{and}\quad v(\beta) = \sum_{\substack{g \in \calM \\ |g| \leq q^N}} e(g\beta).\]
It is not hard to show that
\[G(\ell/f^r+\beta) \sim |f|^{-r} S(f^r,\ell)v(\beta),\]
for $\beta$ sufficiently small. Considering the entirety of the major arcs would then yield
\[\int_{\grM_N(Q)} |G(\alpha)|^k \d\alpha \sim \sum_{\substack{f \in \calM \\ |f| \leq q^{Q/r}}}\mu^2(f)\sum_{\substack{|\ell|<|f|^r \\ (\ell,f^r) \; r\text{-free}}} \dfrac{|S(f^r,\ell)|^k}{|f|^{rk}} \int_{\{|\beta| < q^{Q-N}|f|^{-r}\}} |v(\beta)|^k \d\alpha.\]

Obtaining appropriate completions of the singular integral and the singular series then gives us an asymptotic for the major arcs. We tackle the analysis of the singular series and integral first for convenience, starting with a straightforward evaluation of $S(f^r,\ell)$.
\begin{lemma}\label{S-ident}
    One has that $S(f^r,\ell) \asymp 1$, and moreover, the identity
    \[S(f^r,\ell) = \dfrac{\mu(f)|f|^r}{\zeta_q(r)\Phi_r(f)}\]
    holds.
\end{lemma}
\begin{proof}
    We consider first the sum over $u$, noting that by identity \eqref{sum-ident-non-monic} it is equal to zero unless $\|d^r\ell/f^r\|<|f|^r$, in which case $f^r|(d^r\ell)$. Since we have $(\ell,f^r)$ $r$-free and $f$ is squarefree, we must then have that $f|d$. Thus
    \[S(f^r,\ell) = |f|^r\sum_{\substack{d \in \calM \\ f|d}} \dfrac{\mu(d)}{|d|^r} = \mu(f)\sum_{\substack{d \in \calM \\ (d,f) = 1}} \dfrac{\mu(d)}{|d|^r},\]
    since $\mu(d)$ is supported on squarefrees. The remaining sum is standard, and when $f\neq 1$ is equal to $L(r,\chi_0)^{-1}$, where $\chi_0$ is the trivial character modulo $f$. When $f = 1$, it is equal to $1/\zeta_q(r)$. Finally, observing the identity
    \[L(r,\chi_0) = \zeta_q(r)\prod_{\pi|f} (1-|\pi|^{-r}) = \zeta_q(r)\Phi_r(f)|f|^{-r},\]
    we have the desired result.
\end{proof}

Now define the truncated singular series
\[\grS_r(k;Q) = \sum_{\substack{f \in \calM \\ |f| \leq q^{Q/r}}} \mu^2(f)\Phi_r(f)^{1-k}\]
and the completed singular series
\[\grS_r(k) = \sum_{f \in \calM} \mu^2(f)\Phi_r(f)^{1-k} = \prod_{\pi \in \calM} \bigg(1+(|\pi|^r - 1)^{1-k}\bigg).\]
\begin{lemma}\label{sing-ser-conv}
    The singular series converges absolutely for $k>1+1/r$, and moreover, one has
    \[\grS_r(k) - \grS_r(k;Q) \asymp q^{-Q(k-1-1/r)}.\]
    If $k = 1+1/r$, then $\grS_r(k;Q) \asymp N$, and if $k<1+1/r$, then $\grS_r(k;Q) \asymp q^{Q(1+1/r-k)}$.
\end{lemma}
\begin{proof}
    This follows trivially from the bounds on $\Phi_r(f)$.
\end{proof}
This concludes the necessary analysis of the singular series. We now define the truncated singular integral
\[J_r(k,f;Q) = \int_{\{|\beta| < q^{Q-N}|f|^{-r}\}} |v(\beta)|^k \d\beta\]
and the completed singular integral
\[J_r(k) = \int_{\T} |v(\beta)|^k \d\beta.\]
\begin{lemma}\label{sing-int-conv}
    For $k>1$, one has
    \[J_r(k) - J_r(k,f;Q) \ll q^{(N-Q)(k-1)}|f|^{r(k-1)} .\]
\end{lemma}
\begin{proof}
    We have the standard bound
    \[v(\beta) \ll \min\{q^N,|\beta|^{-1}\}.\]
    Thus 
    \[J_r(k)-J_r(k,f;Q) \ll \int_{\{|\beta| \geq q^{Q-N}|f|^{-r}\}}|\beta|^{-k} \d\beta.\]
    Breaking the region of integration into subsets of size $O(q^{-n})$ where $|\beta|^{-1} = q^n$, we have that 
    \[J_r(k) - J_r(k,f;Q) \ll \sum_{n = 0}^{N-Q+r\deg(f)} q^{n(k-1)},\]
    which yields the desired result.
\end{proof}

We now obtain a somewhat uninformative expression for the singular integral.
\begin{lemma}\label{sing-int}
    Suppose $k>1$. One has
    \[J_r(k) = B(k)q^{N(k-1)}+ o(q^{N(k-1)}),\]
    where
    \[B(k) = \bigg(\dfrac{q^{k-1}}{q^{k-1}-1}\bigg)\bigg(\dfrac{1}{q}\sum_{a \in \F_q} \bigg|e_q(a) + \dfrac{1}{q-1}\bigg|^k-\dfrac{1}{(q-1)^k}\bigg).\]
\end{lemma}
\begin{proof}
    Define $H(\alpha) = \min\{N,-\ord\{\alpha\} - 1\}$. Observe that for $\alpha \in \T$, the identity
    \[\sum_{\substack{g \in \calM \\ |g| \leq q^N}} e(g\alpha) = \sum_{n \leq N} \sum_{g \in \calM_n} e(g\alpha) = \sum_{n = 0}^{H(\alpha)} e(t^n\alpha)q^n\]
    is dependent only on $\ord(\alpha)$ and the leading coefficient of $\alpha$. Decompose the torus into disjoint regions $S(n,a)$ defined by
    \[S(n,a) = \{\beta \in \T: \ord(\beta) = -n-1, \text{ and } \beta_{-n-1} = a\},\]
    and define $T(N) = \{\beta \in \T:\ord(\beta)<-N-1\}$. Then we have that when $\beta \in S(n,a)$,
    \[\sum_{\substack{g \in \calM \\ |g| \leq q^N}} e(g\beta) = e_q(a)q^n + \sum_{0 \leq i < n} q^i = q^n\bigg(e_q(a) + \dfrac{1}{q-1} + O(q^{-n})\bigg),\]
    and for $\beta \in T(N)$ we similarly have
    \[\sum_{\substack{g \in \calM \\ |g| \leq q^N}} e(g\beta) = q^N\bigg(\dfrac{q}{q-1}+O(q^{-N})\bigg).\]

    We can write the torus as the disjoint union of all $S(n,a)$ with $a \in \F_q\setmin\{0\}$ and $0\leq n \leq N$ along with the set $T(N)$. Thus we have that
    \begin{align*}
        J&_r(k) = \sum_{0\leq n\leq N} \sum_{a \in \F_q \setmin\{0\}} \int_{S(n,a)} \bigg|\sum_{\substack{g \in \calM \\ |g| \leq q^N}} e(g\beta)\bigg|^k \d\beta + \int_{T(N)} \bigg|\sum_{\substack{g \in \calM \\ |g| \leq q^N}} e(g\beta)\bigg|^k \d\beta \\
        &= \sum_{0\leq n\leq N} \sum_{a \in \F_q \setmin\{0\}} \int_{S(n,a)}q^{nk}\bigg|e_q(a) + \dfrac{1}{q-1} + O(q^{-n})\bigg|^k \d\beta + \int_{T(N)} q^{Nk}\bigg|\dfrac{q}{q-1}+O(q^{-N})\bigg|^k\d\beta.
    \end{align*}
    Observe that each set $S(n,a)$ is of size $|S(n,a)| = q^{-n-1}$, and the set $T(N)$ is of size $|T(N)| = q^{-N-1}$. Then
    \begin{align*}
        J_r(k) &= \sum_{0\leq n\leq N} q^{n(k-1)} \bigg(\dfrac{1}{q}\sum_{a \in \F_q \setmin\{0\}} \bigg|e_q(a) + \dfrac{1}{q-1} + O(q^{-n})\bigg|^k\bigg) + q^{N(k-1)}\bigg(\dfrac{1}{q}\bigg|\dfrac{q}{q-1}+O(q^{-N})\bigg|^k\bigg) \\
        &=  \sum_{0\leq n\leq N} q^{n(k-1)} \bigg(\dfrac{1}{q}\sum_{a \in \F_q \setmin\{0\}} \bigg|e_q(a) + \dfrac{1}{q-1}\bigg|^k\bigg) + q^{N(k-1)}\bigg(\dfrac{1}{q}\bigg|\dfrac{q}{q-1}\bigg|^k\bigg) + o(q^{N(k-1)}).
    \end{align*}
    We define
    \[A(k) = \dfrac{1}{q}\sum_{a \in \F_q} \bigg|e_q(a) + \dfrac{1}{q-1}\bigg|^k,\]
    and so reorganizing the terms we obtain
    \[I_r(k) = A(k)q^{N(k-1)} + \bigg(A(k) - \dfrac{q^{k-1}}{(q-1)^k}\bigg)\sum_{0 \leq n < N} q^{n(k-1)} + o(q^{N(k-1)}).\]
    Finally, collecting the terms of the geometric sum and simplifying, we obtain the result.
\end{proof}

We are now prepared to handle the major arcs.

\begin{lemma}\label{major-arcs}
    Let $\alpha = \ell/f^r + \beta$, where $|\beta| < q^{Q-N}|f|^{-r}$, and let $G^*(\alpha) = |f|^{-r}S(f^r,\ell)v(\beta)$. Then one has
    \[G(\alpha) - G^*(\alpha) \ll \max\{q^{N/r}|f|^{r-1},q^N|\beta|^{1-1/r}|f|^{r-1}\},\]
    and, as a consequence,
    \[\int_{\grM_N(Q)} |G(\alpha)|^k -  |G^*(\alpha)|^k \d\alpha \ll \begin{cases}
        q^{N(k-2+1/r)} \quad &\text{if } k>3 \\
        Nq^{N(k-2+1/r)} \quad &\text{if } k=3 \\
        Nq^{N(k-2+1/r) + Q(3-k)} \quad &\text{if } k < 3.
    \end{cases}\]
\end{lemma}
\begin{proof}
    Let $L = \min\{q^{N/r}/|f|,|\beta|^{-1/r}/|f|\}$ and $K = \max\{q^{N/r}|f|^{r-1},q^N|\beta|^{1-1/r}|f|^{r-1}\}$. First observe that
    \[\sum_{\substack{d \in \calM \\ |d| \geq L}} \mu(d) \sum_{\substack{g \in \calM \\ |g| \leq q^N/|d|^r}}e(gd^r\alpha) \ll q^N \sum_{\substack{d \in \calM \\ |d| \geq L}} |d|^{-r}  \ll K.\]
    Thus we may consider the sum over $|d|<L$. Sorting the sum over $g$ by congruence modulo $f^r$, we have
    \[G(\alpha) = \sum_{\substack{d \in \calM \\ |d| < L}} \mu(d) \sum_{|u| < |f|^r}e(ud^r\alpha)\sum_{\substack{g \in \calM \\ |g| \leq q^N/|df|^r}} e(gf^rd^r\beta) + O(K).\]
    
    Since $|f^r\beta|<q^{Q-N}$ and $|d| < L$, we have that $|ud^r\beta|<q^{-1}$ for any $|u|<|f|^r$, and so we can simplify the sum over $u$ to get
    \[G(\alpha) = \sum_{\substack{d \in \calM \\ |d| < L}} \mu(d) \sum_{|u| < |f|^r}e(ud^r\ell/f^r)\sum_{\substack{g \in \calM \\ |g| \leq q^N/|df|^r}} e(gf^rd^r\beta) + O(K).\]
    Similarly, if $|u|<|f^rd^r|$, then $|u\beta| < q^{-1}$, and so we can write the sum over $g$ as an average, obtaining
    \begin{align*}
        G(\alpha) &= |f|^{-r}\bigg(\sum_{\substack{d \in \calM \\ |d| < L}} \dfrac{\mu(d)}{|d|^r} \sum_{|u| < |f|^r}e(ud^r\ell/f^r)\bigg)\bigg(\sum_{\substack{g \in \calM \\ |g| \leq q^N}} e(g\beta) \bigg)+ O(K).
    \end{align*}
    Finally, we observe that we can complete the sum over $d$ with an error term of the same order as the existing one, obtaining
    \[G(\alpha) = |f|^{-r}\bigg(\sum_{d \in \calM} \dfrac{\mu(d)}{|d|^r} \sum_{|u| < |f|^r}e(ud^r\ell/f^r)\bigg)\bigg(\sum_{\substack{g \in \calM \\ |g| \leq q^N}} e(g\beta) \bigg)+ O(K).\]
    This is the first result.

    Now we consider integration over a single major arc. Recall the definitions of $S(n,a)$ and $T(N)$ from the previous lemma, and define $S(n)$ to be the union of $S(n,a)$ for all $a \in \F_q\setmin\{0\}$. We can write
    \[\int_{S(n)} |G(\alpha)|^k -  |G^*(\alpha)|^k \d\alpha \ll |f|^{r(1-k)}\int_{S(n)} |v(\beta)^{k-1}K| \d\alpha.\]
    If $n<N$, then $K = q^N|\beta|^{1-1/r}|f|^{r-1}$, and so 
    \[\int_{S(n)} |G(\alpha)|^k -  |G^*(\alpha)|^k \d\alpha \ll q^{N+n(1/r-1)}|f|^{r(2-1/r-k)}\int_{S(n)} |v(\beta)|^{k-1} \d\alpha.\]
    Clearly in $S(n)$ we have that $v(\beta) \ll q^n$, and so 
    \[\int_{S(n)} |G(\alpha)|^k -  |G^*(\alpha)|^k \d\alpha \ll q^{N+n(k-3+1/r)}|f|^{r(2-1/r-k)}.\]
    The same result holds for $T(N)$. Summing over $N+r\deg(f)-Q \leq n < N$ and $T(N)$, we have that
    \[\int_{\grM_N(f^r,\ell)} |G(\alpha)|^k -  |G^*(\alpha)|^k \d\alpha \ll \begin{cases}
        q^{N(k-2+1/r)}|f|^{r(2-1/r-k)} \quad &\text{if } k>3-1/r \\
        Nq^N|f|^{-r} \quad &\text{if } k=3-1/r \\
        q^{N(k-2+1/r)-Q(k-3+1/r)}|f|^{-r} \quad &\text{if } k < 3-1/r.
    \end{cases}\]
    Thus over the entirety of the major arcs, we have that
    \[\int_{\grM_N(Q)} |G(\alpha)|^k -  |G^*(\alpha)|^k \d\alpha \ll \begin{cases}
        q^{N(k-2+1/r)} \quad &\text{if } k>3 \\
        Nq^{N(k-2+1/r)} \quad &\text{if } k=3 \\
        Nq^{N(k-2+1/r) + Q(3-k)} \quad &\text{if } k < 3.
    \end{cases}\]
    This is the desired result. We remark that the $N$ in the case $k<3$ is unnecessary except for $k=3-1/r$, but we write the result this way for convenience.
\end{proof}

We now take a brief moment to analyze the minor arcs. A relatively effortless pointwise bound is enough for our needs.
\begin{lemma}
    For all $\alpha \in \grm_N(Q)$, one has
    \[G(\alpha) \ll q^{N-Q(1-1/r)}.\]
\end{lemma}
\begin{proof}
    If $\alpha \in \grm_N(Q)$, then for each $|d| \leq q^{Q/r}$, we have $\|d^r\alpha\| \geq q^{Q-N}$, and so 
    \[\sum_{\substack{g \in \calM \\ |g| \leq q^N/|d|^r}} e(gd^r\alpha) \ll \min\{q^{N-Q},q^N/|d|^r\}.\]
    Thus
    \[G(\alpha) \ll \sum_{|d|\leq q^{N/r}}\min\{q^{N-Q},q^N/|d|^r\} \ll q^{N-Q(1-1/r)}. \qedhere\]
\end{proof}
Now we can simply apply the upper bound for $k=1+1/r$ to yield
\[\int_{\grm_N(Q)} |G(\alpha)|^k \d\alpha \leq \bigg(\sup_{\alpha \in \grm_N(Q)} |G(\alpha)|\bigg)^{k-1-1/r}Nq^{N/r} \ll Nq^{N(k-1)-Q(1-1/r)(k-1-1/r)}.\]

\asympt*
Here the constants may be taken to be
\[C = \bigg(\dfrac{1}{\zeta_q(r)}\bigg)^k \grS_r(k)B(k)\]
and
\[\theta = \begin{cases}
    1-1/r \quad &\text{if } k>3+1/r \\
    \frac{1}{2}(1-1/r)(k-1-1/r) - \epsilon \quad &\text{if } 2+1/r < k \leq 3+1/r \\
    \frac{(1-1/r)^2(k-1-1/r)}{(1-1/r)(k-1-1/r)+(3-k)} - \epsilon \quad &\text{if } k\leq 2+1/r.
\end{cases}\]
\begin{proof}
    First consider the major arc contribution. We have by Lemma \ref{major-arcs} that
    \[\int_{\grM_N(Q)} |G(\alpha)|^k \d\alpha = \int_{\grM_N(Q)} |G^*(\alpha)|^k \d\alpha + E_1,\]
    where
    \[E_1 \ll \begin{cases}
        q^{N(k-2+1/r)} \quad &\text{if } k>3 \\
        Nq^{N(k-2+1/r)} \quad &\text{if } k=3 \\
        q^{N(k-2+1/r) + Q(3-k)+\epsilon} \quad &\text{if } k < 3.
    \end{cases}\]
    We can evaluate the main term above to yield
    \[\sum_{\substack{f \in \calM \\ |f| \leq q^{Q/r}}}\mu^2(f)\sum_{\substack{|\ell|<|f|^r \\ (\ell,f^r) \; r\text{-free}}} \dfrac{|S(f^r,\ell)|^k}{|f|^{rk}}J_r(k,f;Q) + E_1,\]
    and applying Lemmas \ref{sing-int-conv}, and \ref{sing-int}, we can complete the singular integral to obtain
    \[\int_{\grM_N(Q)} |G(\alpha)|^k \d\alpha = \bigg(\sum_{\substack{f \in \calM \\ |f| \leq q^{Q/r}}}\mu^2(f)\sum_{\substack{|\ell|<|f|^r \\ (\ell,f^r) \; r\text{-free}}} \dfrac{|S(f^r,\ell)|^k}{|f|^{rk}}\bigg)B(k)q^{N(k-1)} + E_1 + E_2,\]
    where $E_2 \ll q^{N(k-1)-Q(k-1-1/r)}$, since $k>1+1/r$. Similarly, applying Lemmas \ref{S-ident} and \ref{sing-ser-conv}, we can complete the singular series to obtain
    \[\int_{\grM_N(Q)} |G(\alpha)|^k \d\alpha = \bigg(\dfrac{1}{\zeta_q(r)}\bigg)^k \grS_r(k)B(k)q^{N(k-1)} + E_1 + E_2.\]
    Finally, we remark that the minor arc contribution is $O(E_2)$, and so 
    \[\int_{\T} |G(\alpha)|^k \d\alpha = \bigg(\dfrac{1}{\zeta_q(r)}\bigg)^k \grS_r(k)B(k)q^{N(k-1)} + E_1 + E_2.\]
    The result follows from comparing the error terms $E_1$ and $E_2$. The term $E_1$ dominates provided $k>3+1/r$. When $2+1/r < k \leq 3+1/r$, the term $E_2$ dominates, and for the remaining case $Q$ is chosen to optimize the error term.
\end{proof}

\section*{Acknowledgements}

The author is grateful to his advisor Trevor Wooley for his encouragement and insights, and in particular for his suggestions regarding Lemma \ref{crit-upper}. He also wishes to thank Atal Bhargava and Anurag Sahay for their helpful conversation and advice.

\bibliography{refs}

\end{document}